\documentclass[12pt, reqno]{amsart}

\usepackage{amsmath, amstext, amsbsy, amssymb, amscd}
\usepackage[mathscr]{eucal}

\usepackage{tikz}
\usepackage{tikz-cd}
\usepackage{tikz-3dplot}
\usepackage{verbatim}
\usepackage{algorithm}
\usepackage{algpseudocode}
\usepackage{courier}
\usetikzlibrary{arrows,shapes} 
\usepackage{listings}

\makeatletter
\@namedef{subjclassname@2020}{%
  \textup{2020} Mathematics Subject Classification}
\makeatother

\newtheorem{theorem}{Theorem}[section]
\newtheorem{lemma}[theorem]{Lemma}
\newtheorem{proposition}[theorem]{Proposition}
\newtheorem{corollary}[theorem]{Corollary}
\theoremstyle{definition}
\newtheorem{definition}[theorem]{Definition}
\newtheorem{example}[theorem]{Example}

\theoremstyle{remark}
\newtheorem{remark}[theorem]{Remark}

\newcommand{\A}{ \mathbb A}

\newcommand{\C}{ \mathbb C}

\newcommand{\e}{{\bf e}}

\newcommand{\gr}{{\rm gr}}
\newcommand{\Hilb}{ {\rm Hilb}}
\newcommand{\id}{ {\rm Id}}
\newcommand{\im}{ {\rm im}}
\newcommand{\kbf}{ {\bf k}}
\newcommand{\la}{ {\lambda}}

\newcommand{\N}{ \mathbb N}

\newcommand{\R}{ \mathbb R}

\newcommand{\T}{ \mathbb T}

\newcommand{\W}{\widetilde}

\newcommand{\Z}{ \mathbb Z}

\def\beq{\begin{equation}}
\def\eeq{\end{equation}}

\numberwithin{equation}{section}

\begin{document}
\title[$\N^d$-indexed persistence modules, partitions and rank invariants]
{$\N^d$-indexed persistence modules, higher dimensional partitions and rank invariants}

\author[Mehdi Nategh]{Mehdi Nategh
}
\address{Department of Mathematics, University of Missouri, Columbia, MO 65211, USA}
\email{nateghm@umsystem.edu}

\author[Zhenbo Qin]{Zhenbo Qin
}
\address{Department of Mathematics, University of Missouri, Columbia, MO 65211, USA}
\email{qinz@missouri.edu}

\author[Shuguang Wang]{Shuguang Wang
}
\address{Department of Mathematics, University of Missouri, Columbia, MO 65211, USA}
\email{wangs@missouri.edu}

\date{\today}

\begin{abstract}
We study decomposable $\N^d$-indexed persistence modules via 
higher dimensional partitions. Their barcodes are defined in terms of 
the extended interior of the corresponding Young diagrams. 
For two decomposable $\N^d$-indexed persistence modules, we present 
a necessary and sufficient condition, in terms of the partitions, 
for their rank invariants to be the same. 
This generalizes the well-known fact that for an $\N$-indexed persistence module, 
its barcode and its rank invariant determine each other, i.e., 
the rank invariant is a complete invariant.
\end{abstract}

\maketitle
\section{\bf Introduction}   
\label{section_introduction}

A fundamental structure theorem asserts that every one-dimensional persistence module 
admits a unique decomposition up to isomorphisms \cite{CV, EH}. 
This gives rise to the concept of barcodes which plays a pivotal role in 
topological data analysis. It is known from the pioneering work \cite{CZ} that 
such structure theorem is no longer true for higher dimensional persistence modules. 
Extensive research has been devoted to higher dimensional persistence modules 
in recent years (see \cite{AL, BR, KNCB, Les, LS, Vip} and the references therein). 
Higher dimensional persistence modules have found important applications 
in the study of noisy point cloud data and time-varying data.

In this paper, we study $\N^d$-indexed persistence modules over a field $k$ via 
$d$-dimensional partitions, where $\N$ denotes the set of nonnegative integers. 
To motivate our concepts to be introduced below, let us look at the case $d=1$. 
It is well-known that the barcode of an $\N$-indexed persistence module over $k$ 
is a multiset consisting of some intervals of the form $[a, b) = \T_a [0, b-a)$ 
where $a \in \N$, $b \in \N \sqcup \{+\infty\}$, 
and $\T_a: \N \to \N$ is the translation by $a$. 
The closed interval $[0, b-a]$ is precisely the Young diagram of 
the $1$-dimensional partition $(b-a)_O$ corresponding to $b-a$ 
where $O \in \N^0 = \{ O \}$, while the interval $[0, b-a)$ 
may be regarded as the extended interior of the Young diagram $[0, b-a]$.
 
For a general integer $d \ge 1$, a {\it $d$-dimensional partition} $\lambda$ is 
an array 
$$
\lambda = (\lambda_{i_1, \ldots, i_{d-1}})_{i_1, \ldots, i_{d-1}}
$$
of $\lambda_{i_1, \ldots, i_{d-1}} \in \N \sqcup \{+\infty\}$ indexed by 
$(i_1, \ldots, i_{d-1}) \in \N^{d-1}$ such that 
$$
\lambda_{i_1, \ldots, i_{d-1}} \ge \lambda_{j_1, \ldots, j_{d-1}}
$$
if $i_1 \le j_1, \ldots, i_{d-1} \le j_{d-1}$. 
For a $d$-dimensional partition $\la$, the {\it extended interior} of 
its Young diagram $D_\la \subset (\R^+)^d$ is the region
$$
D_\la^{\rm ei} = D_\la - \overline{(\partial D_\la) \cap (\R^+)^d}.
$$
Put $D_\la^{\rm int} = D_\la^{\rm ei} \cap \N^d$ which is the set of integral points 
in $D_\la^{\rm ei}$. Geometrically, $D_\la^{\rm ei}$ is obtained from 
the Young diagram $D_\la$ by removing its boundary in $(\R^+)^d$, 
and $D_\la^{\rm int}$ consists of all the integral points in $D_\la^{\rm ei}$.

\begin{example}   \label{example-Dlaei}
For the $2$-dimensional partitions $\la=(3, 3, 1)$, the extended interior 
$D_\la^{\rm ei}$ of $D_\la$ is illustrated by Figure~1 below. Note that 
$$
D_\la^{\rm int} = \{(0,0), (1, 0), (2, 0), (0,1), (1, 1), (0, 2), (1, 2)\}
$$ 
which consists of all the integral points in $D_\la^{\rm ei}$.
\begin{center}
	\begin{tikzpicture}[x=0.5cm,y=0.5cm,step=0.5cm]
		\fill[blue!30] (0,3) -- (2,3) -- (2,1) -- (3,1) -- (3,0) -- (0,0) -- cycle;
		\draw[->] (0,0) -- (4,0);
		\draw[->] (0,0) -- (0,4);
		=========================
		\draw[dashed] (0,3) -- (2,3);
		\draw[dashed] (2,3) -- (2,1);
		\draw[dashed] (2,1) -- (3,1);
		\draw[dashed] (3,1) -- (3,0);
		\filldraw[fill=white] (0,3) circle (2pt);
		\filldraw[fill=white] (3,0) circle (2pt);
		
		\node at (-0.4,-0.4) {0};
		\foreach \x in {1,2}
		{
			\node[fill,circle, inner sep=1pt,label=below:$\x$] at (\x,0) {};
			\node[fill, circle, inner sep=1pt,label=left:$\x$] at (0,\x) {};
		}
		\node[circle, inner sep=1pt, label=left:3] at (0,3) {};
		\node[circle, inner sep=1pt, label=below:3] at (3,0) {};
	\end{tikzpicture}

\end{center}
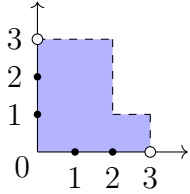
\centerline{\footnotesize \text{\bf Figure 1. } \text{$D_\la^{\rm ei}$ for a 
		2-dimensional partition of $ n = 7 $.}}
\end{example}
 
Define the $\N^d$-indexed persistence module $\kbf_\la$ by
$$
  (\kbf_\la)_x 
= \begin{cases}
	k &\text{if $x \in D_\la^{\rm int}$}, \\0 &\text{otherwise}.
  \end{cases}    
$$
For $x, y \in \N^d$ with $x \le y$, the morphism $(\kbf_\la)_{x, y}: 
(\kbf_\la)_x \to (\kbf_\la)_y$ is the identity map $\id_k$ if 
$x,y \in D_\la^{\rm int}$ and is $0$ otherwise.

We define that an $\N^d$-indexed persistence module $M$ admits a {\it barcode} if 
$$ 
M \cong \bigoplus_{i \in \Lambda} \T_{a^{(i)}}\kbf_{\la^{(i)}}
$$
where the index set $\Lambda$ is finite, 
and for each $i \in \Lambda$, $a^{(i)} \in \N^{d}$, 
$\T_{a^{(i)}}: \N^d \to \N^d$ is the translation associated to $a^{(i)}$, and 
$\la^{(i)}$ is a $d$-dimensional partition with nonzero size $|\la^{(i)}| \ne 0$. 
In this case, the {\it barcode} ${\mathfrak B}_M$ of $M$ is defined 
to be the multiset whose elements are
$$
\T_{a^{(i)}} (D^{\rm ei}_{\la^{(i)}}), \quad i \in \Lambda.
$$
Criterions and algorithms for determining whether 
an $\N^d$-indexed persistence module 
admits a barcode have been investigated in \cite{Nat}. 

Since a general higher dimensional persistence module may not admit a barcode, 
the rank invariant was introduced in \cite{CZ} as an alternative discrete invariant.
The rank invariant of an $\N^d$-indexed persistence module $M$ is 
the function ${\rm Rank}^M: (\N^d)^{\le} \to \N$ given by
$$
{\rm Rank}^M(x, y) = \text{\rm Rank}(M_{x,y})
$$
where $(\N^d)^{\le} = \{(x,y) \in \N^d \times \N^d|x \le y\}$. 
Indeed, Carlsson and Zomorodian \cite[Theorem~12]{CZ} proved that 
when $d=1$, the barcode and the rank invariant determine each other, i.e., 
the rank invariant is a complete invariant. 
However, when $d > 1$, no prior necessary and sufficient conditions 
for determining the rank invariant are known. 
Our main result in this paper provides a necessary and sufficient condition 
(in terms of the parts in the partition) for determining the rank invariant when 
the $\N^d$-indexed persistence module admits a barcode. 
When $d=1$, our necessary and sufficient condition exactly says that 
the barcode and the rank invariant determine each other.

\begin{theorem}  \label{Intro_ThmRk-inv-MN}
Let $d \ge 1$. 
Let $M$ and $N$ be $\N^d$-indexed persistence modules admitting barcodes:
$$
M = \bigoplus_{i \in \Lambda_1} \T_{a^{(i)}}\kbf_{\la^{(i)}}
\qquad
\text{and}
\qquad
N = \bigoplus_{\ell \in \Lambda_2} \T_{b^{(\ell)}}\kbf_{\mu^{(\ell)}}
$$
where $|\la^{(i)}| \ne 0$ and $|\mu^{(\ell)}| \ne 0$ for all 
$i \in \Lambda_1$ and $\ell \in \Lambda_2$. Then, ${\rm Rank}^M = {\rm Rank}^N$ 
if and only if for every $(i_1, \ldots, i_{d-1}) \in \N^{d-1}$, the two multisets 
\begin{eqnarray}    \label{Intro_ThmRk-inv-MN.01}
\Big \{\big (a^{(i)}, (\la^{(i)})_{i_1, \ldots, i_{d-1}} \big )
   |i \in \Lambda_1 \text{ and } (\la^{(i)})_{i_1, \ldots, i_{d-1}} > 0 \Big \}
\end{eqnarray}
and 
\begin{eqnarray}    \label{Intro_ThmRk-inv-MN.02}
\Big  \{ \big (b^{(\ell)}, (\mu^{(\ell)})_{i_1, \ldots, i_{d-1}} \big )
   |\ell \in \Lambda_2 \text{ and } 
   (\mu^{(\ell)})_{i_1, \ldots, i_{d-1}} > 0 \Big \}
\end{eqnarray}
are equal. 
\end{theorem}

The main idea in the proof of Theorem~\ref{Intro_ThmRk-inv-MN} is to use 
induction on the sizes of $M$ and $N$. We remark that when $d > 1$, 
under the conditions of Theorem~\ref{Intro_ThmRk-inv-MN}, 
${\rm Rank}^M = {\rm Rank}^N$ does not imply that $M$ and $N$ have the same barcode. 
In other words, when $d > 1$, the rank invariant is not a complete invariant 
for decomposable $\N^d$-indexed persistence modules. 
It would be interesting to see how to strengthen the assumption 
${\rm Rank}^M = {\rm Rank}^N$ in Theorem~\ref{Intro_ThmRk-inv-MN} so that 
the decomposable $\N^d$-indexed persistence modules $M$ and $N$ are 
guaranteed to have the same barcode. 

The paper is organized as follows. In Section~\ref{section_Higher}, 
higher dimensional partitions and Young diagrams are reviewed. 
We define $d$-dimensional barcodes via the extended interiors of Young diagrams. 
Section~\ref{section_MPP} is devoted to $\N^d$-indexed persistence modules. 
In Section~\ref{section_rank}, we prove Theorem~\ref{Intro_ThmRk-inv-MN} 
(= Theorem~\ref{ThmRk-inv-MN}).

\section{\bf Higher dimensional partitions and barcodes}   
\label{section_Higher}

\begin{definition}  \label{d-part}
Let $\N$ be the set of {\it nonnegative} integers. 
Let $d \ge 1$ be an integer. 

\begin{enumerate}
\item[{\rm (i)}] 
When $d \ge 2$, a {\it $d$-dimensional partition} $\lambda$ is 
an array 
$$
\lambda = (\lambda_{i_1, \ldots, i_{d-1}})_{i_1, \ldots, i_{d-1}}
$$
of $\lambda_{i_1, \ldots, i_{d-1}} \in \N \sqcup \{+\infty\}$ indexed by 
$(i_1, \ldots, i_{d-1}) \in \N^{d-1}$ such that 
$$
\lambda_{i_1, \ldots, i_{d-1}} \ge \lambda_{j_1, \ldots, j_{d-1}}
$$
if $i_1 \le j_1, \ldots, i_{d-1} \le j_{d-1}$. For $n \in \N \sqcup \{+\infty\}$, 
define the unique {\it $1$-dimensional partition} of $n$ to be 
$\lambda = (n)_O$ indexed by $O \in \N^0 = \R^0 = \{O\}$.

\item[{\rm (ii)}] 
The {\it size} $|\la|$ of a partition $\la$ is defined to be 
$
|\la| = \sum_{i_1, \ldots, i_{d-1}} \lambda_{i_1, \ldots, i_{d-1}}.
$
If $|\la| = n \in \N \sqcup \{+\infty\}$, then $\la$ is called
a {\it partition of $n$} and denoted by $\lambda \vdash n$. 

\item[{\rm (iii)}] 
For $n \in \N \sqcup \{+\infty\}$, the set of $d$-dimensional partitions 
of $n$ is denoted by $\mathcal P_d(n)$. Define $P_d(n)$ to be the number of $d$-dimensional partitions of $n$. 
\end{enumerate}
\end{definition}

\begin{remark}  \label{d=2}
The ordinary partitions are $2$-dimensional partitions 
(of nonnegative integers) in our sense.
\end{remark} 

One sees immediately that the generating function for $P_2(n)$ is given by
\begin{eqnarray}   \label{P2n}
	\sum_{n = 0}^{+\infty} P_2(n) q^n = \prod_{n = 1}^{+\infty}  \frac{1}{1-q^n}
\end{eqnarray}
where $q$ is a formal variable. A well-known result of McMahon \cite{And} states that
\begin{eqnarray}   \label{P3n}
	\sum_{n = 0}^{+\infty} P_3(n) q^n 
	\,\, = \,\, \prod_{n = 1}^{+\infty}  \frac{1}{(1-q^n)^n}.
\end{eqnarray}
There is no analogous formula for $P_d(n)$ when $d > 3$.

Fix $n \in \N$ and the field $k = \C$. The group $(k^*)^d$ acts 
on $k[t_1, \ldots, t_d]$ via 
$$
(k_1, \ldots, k_d) (t_1, \ldots, t_d) = (k_1t_1, \ldots, k_dt_d)
$$
where $k^* = k-\{0\}$ and $(k_1, \ldots, k_d) \in (k^*)^d$. 
It induces a $(k^*)^d$-action 
on the Hilbert scheme $\Hilb^n(\A_k^d)$ parametrizing length-$n$ 
$0$-dimensional closed subschemes of $\A_k^d = 
{\rm Spec} \, k[t_1, \ldots, t_d]$ (see \cite{Qin}). A $d$-dimensional partition 
$\lambda = (\lambda_{i_1, \ldots, i_{d-1}})_{i_1, \ldots, i_{d-1}}$ of 
$n$ determines a $(k^*)^d$-invariant ideal 
$$
  I 
= \big \langle t_1^{i_1} \cdots t_{d-1}^{i_{d-1}} 
  t_d^{\lambda_{i_1, \ldots, i_{d-1}}}| 
  (i_1, \ldots, i_{d-1}) \in \N^{d-1} \big \rangle 
$$
of $k[t_1, \ldots, t_d]$
with co-length $n$ (i.e, $\dim_k k[t_1, \ldots, t_d]/I = n$). 
In this way, the set of $d$-dimensional partitions of $n$ is in 
one-to-one correspondence with the set of $(k^*)^d$-invariant 
ideals of $k[t_1, \ldots, t_d]$ with co-length $n$, which in turn is 
in one-to-one correspondence with the set of $(k^*)^d$-invariant 
points in $\Hilb^n(\A_k^d)$.

\begin{definition}  \label{Youngdiagram}
	Let $\la$ be a $d$-dimensional partition and $a = (a_1, \ldots, a_d) 
	\in \N^d$.
	\begin{enumerate}
		\item[{\rm (i)}] 
		The {\it Young diagram $D_\la$} of 
		$\lambda$ is the subset of $\R^d$ obtained by stacking 
		$\lambda_{i_1, \ldots, i_{d-1}}$ many $d$-dimensional unit boxes 
		over the $(d-1)$-dimensional rectangular region in $\R^{d-1} \subset 
		\R^d$ spanned by the vertices 
		$$
		(i_1, \ldots, i_{d-1}), (i_1, \ldots, i_{d-1}) + \e_1, \ldots, 
		(i_1, \ldots, i_{d-1}) + \e_{d-1}
		$$
		where $\{\e_1, \ldots, \e_{d-1}\}$ denotes the standard basis of 
		$\R^{d-1}$ and $\R^{d-1}$ is embedded in $\R^d$ by taking 
		the last coordinate to be $0$.
		
		\item[{\rm (ii)}] The {\it set of integral points inside $D_\la$} is 
		defined by
		\begin{eqnarray}   \label{Youngdiagram.1}
			D_\la^{\rm int} 
			= \{(i_1, \ldots, i_{d-1}, h) \in \N^d|h < \lambda_{i_1, \ldots, i_{d-1}}\}
			\subset D_\la. 
		\end{eqnarray} 
		The {\it extended interior} of $D_\la$ is defined to be
		\begin{eqnarray}   \label{Youngdiagram.2}
			D_\la^{\rm ei} 
			= D_\la - \overline{(\partial D_\la) \cap (\R^+)^d}.
		\end{eqnarray}
		
		\item[{\rm (iii)}] 
		Define $\T_a: \R^d \to \R^d$ to be the translation of $\R^d$ by $a$.
	\end{enumerate}
\end{definition}

\begin{remark}   \label{Dlaint-ei}
For a $d$-dimensional partition $\la$, we have 
$D_\la^{\rm int} = D_\la^{\rm ei} \cap \N^d$.
\end{remark}

When $d=2$ and $\lambda \vdash n \in \N$, $D_\la$ is the usual Young 
diagram (up to some rotation) of the $2$-dimensional (i.e, usual) 
partition $\la$. When $\la = (+\infty)_{i_1, \ldots, i_{d-1}}$ 
(i.e, $\la_{i_1, \ldots, i_{d-1}} = +\infty$ for every 
$(i_1, \ldots, i_{d-1}) \in \N^{d-1}$), we have $D_\la = \R^d$.

\begin{example}
Recall from Definition~\ref{d-part}~(i) the notation $(n)_O$ for 
$n \in \N \sqcup \{+\infty\}$.
In Figure~2 and Figure~3 below, we present two sets of examples of 
$\T_a(D_{\la})$ for $d=1$ and $d=2$ respectively.
\begin{eqnarray*}
	&
	\begin{tikzpicture}[x=0.5cm,step=0.5cm]
		\draw[-, very thick, blue] (0,0) -- (3,0);
		\draw[->] (3,0) -- (7,0);
		\foreach \x in {0,1,2,3,4,5,6}
		{
			\node[fill,circle,inner sep=0.5pt,label=below:$\x$] at (\x,0) {};
		}
		\fill[blue] (0,0) circle[radius=2pt]; 
		\draw[blue] (3,0) circle[radius=2pt]; 
	\end{tikzpicture}
	\qquad
	\begin{tikzpicture}[x=0.5cm,step=0.5cm]
		\draw[-] (0,0) -- (2,0);
		\draw[-, very thick, blue] (2,0) -- (5,0);
		\draw[->] (5,0) -- (7,0);
		\foreach \x in {0,1,2,3,4,5,6}
		{
			\node[fill,circle,inner sep=0.5pt,label=below:$\x$] at (\x,0) {};
		}
		\fill[blue] (2,0) circle[radius=2pt]; 
		\draw[blue] (5,0) circle[radius=2pt]; 
	\end{tikzpicture}
	\qquad
	\begin{tikzpicture}[x=0.5cm,step=0.5cm]
		\draw[-] (0,0) -- (2,0);
		\draw[->, very thick, blue] (2,0) -- (7,0);
		\draw[->] (2,0) -- (6,0);
		\foreach \x in {0,1,2,3,4,5,6}
		{
			\node[fill,circle,inner sep=0.5pt,label=below:$\x$] at (\x,0) {};
		}
		\fill[blue] (2,0) circle[radius=2pt]; 
	\end{tikzpicture}
	&  
\end{eqnarray*}
\centerline{\footnotesize \text{{\bf Figure 2.} 
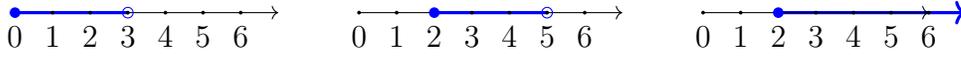
		For the $1$-dimensional partitions 
		$\la=(3)_O$ and $\mu=(+\infty)_O$,}}
\centerline{\footnotesize \text{the left is $D_\la$, 
		the middle is $\T_2(D_\la)$ and the right is $\T_2(D_\mu)$.}}

\begin{eqnarray*}
	&
	\begin{tikzpicture}[x=0.5cm,y=0.5cm,step=0.5cm]
		\draw[->] (0,0) -- (7,0);
		\draw[->] (0,0) -- (0,7);
		\draw[-] (0,3) -- (1,3);
		\draw[-] (1,3) -- (1,1);
		\draw[-] (1,1) -- (2,1);
		\draw[-] (2,1) -- (2,0);
		\draw[fill=blue!30] (0,0) -- ++(0,3) -- ++(1,0) -- ++(0,-2) -- ++(1,0) 
		-- ++(0,-1) -- ++(-2,0);
		\node at (-0.4,-0.4) {0};
		\foreach \x in {1,2,3,4,5,6}
		{
			\node[fill,circle,inner sep=1pt,label=below:$\x$] at (\x,0) {};
			\node[fill,circle,inner sep=1pt,label=left:$\x$] at (0,\x) {};
		}
	\end{tikzpicture}
	\qquad
	\begin{tikzpicture}[x=0.5cm,y=0.5cm,step=0.5cm]
		\fill [blue!30] (0,0) rectangle (1,7);
		\fill [blue!30] (1,0) rectangle (2,4);
		\fill [blue!30] (2,0) rectangle (3,2);
		\fill [blue!30] (3,0) rectangle (7,1);
		\draw[->] (0,0) -- (7,0);
		\draw[->] (0,0) -- (0,7);
		\draw[dashed] (3,1) -- (1,1) -- (1,4);
		\draw[-] (7,1) -- (3,1) -- (3,2) -- (2,2) -- (2,4) -- (1,4) -- (1,7);
		\node at (-0.4,-0.4) {0};
		\foreach \x in {1,2,3,4,5,6}
		{
			\node[fill,circle,inner sep=1pt,label=below:$\x$] at (\x,0) {};
			\node[fill,circle,inner sep=1pt,label=left:$\x$] at (0,\x) {};
		}
	\end{tikzpicture}
	\qquad
	\begin{tikzpicture}[x=0.5cm,y=0.5cm,step=0.5cm]
		\fill [blue!30] (2,1) rectangle (3,7);
		\fill [blue!30] (3,1) rectangle (4,5);
		\fill [blue!30] (4,1) rectangle (5,3);
		\fill [blue!30] (5,1) rectangle (7,2);
		\draw[->] (0,0) -- (7,0);
		\draw[->] (0,0) -- (0,7);
		\draw[-] (7,1) -- (2,1) -- (2,7);
		\draw[dashed] (5,2) -- (3,2) -- (3,5);
		\draw[-] (7,2) -- (5,2) -- (5,3) -- (4,3) -- (4,5) -- (3,5) -- (3,7);
		\node at (-0.4,-0.4) {0};
		\foreach \x in {1,2,3,4,5,6}
		{
			\node[fill,circle,inner sep=1pt,label=below:$\x$] at (\x,0) {};
			\node[fill,circle,inner sep=1pt,label=left:$\x$] at (0,\x) {};
		}
	\end{tikzpicture}
	&  
\end{eqnarray*}
\centerline{\footnotesize \text{{\bf Figure 3.} 
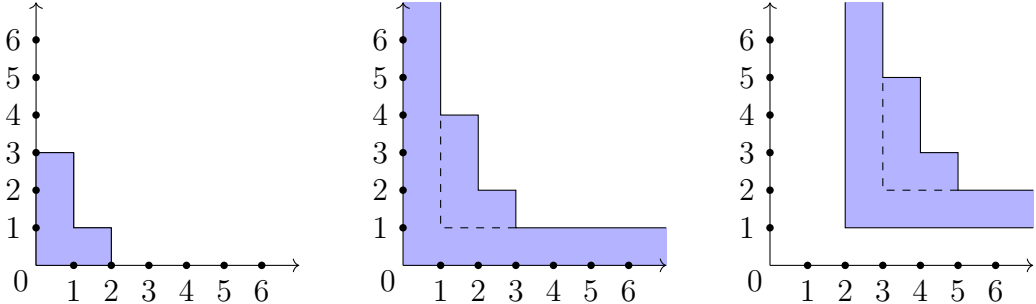
For the $2$-dimensional partitions 
$\la=(3, 1)$ and $\mu = (+\infty, 4,2,1,1, \ldots)$,}}
\centerline{\footnotesize \text{the left is $D_\la$, 
the middle is $D_\mu$ and the right is $\T_{(2,1)}(D_\mu)$.}}
\end{example}

\begin{example}
Figure 4 illustrates Young diagram of a 3-dimensional partition 
of $ n = 29 $ as  $ \la = (\la_{(a,b)}) $, $ (a,b) \in  \N^{3} $, where $ \la_{(a,b)} $ is defined by
\begin{equation}
	\la_{(a,b)} =
	\begin{cases}
		4,  & (a,b) = (0,0) \\
		3,  & (a,b) \in \{ (0,1), (1,0), (1,1), (2,0) \} \\
		2,  & (a,b) \in \{(0,2), (1,2), (2,1), (3,0), (3,1) \} \\
		1.  & (a,b) \in \{ (4,0), (1,3), (0,3) \}
	\end{cases}
\end{equation}
\begin{eqnarray*}
	&	
	\tdplotsetmaincoords{70}{110} 
	\begin{tikzpicture}[tdplot_main_coords,scale=1]
		
		\draw[->] (0,0,0) -- (6,0,0) node[anchor=north east]{$x$};
		\draw[->] (0,0,0) -- (0,5,0) node[anchor=north west]{$y$};
		\draw[->] (0,0,0) -- (0,0,5) node[anchor=south]{$z$};
		
		\foreach \x in {1,5}
		{
			\foreach \y in {1,4}
			{
				\foreach \z in {1}
				{
					\node at (\x,\y,\z) {};
				}
			}
		}
		
		\foreach \x in {0,1,2,3,4}
		{
			\foreach \y in {0}
			{
				\foreach \z in {0}
				{
					\draw[fill=blue!30] (\x,\y,\z) -- (\x+1,\y,\z) -- (\x+1,\y+1,\z) -- (\x,\y+1,\z) -- cycle; 
					\draw[fill=blue!30] (\x,\y,\z) -- (\x,\y+1,\z) -- (\x,\y+1,\z+1) -- (\x,\y,\z+1) -- cycle; 
					\draw[fill=blue!30] (\x,\y,\z) -- (\x+1,\y,\z) -- (\x+1,\y,\z+1) -- (\x,\y,\z+1) -- cycle; 
					\draw[fill=blue!30] (\x+1,\y,\z) -- (\x+1,\y+1,\z) -- (\x+1,\y+1,\z+1) -- (\x+1,\y,\z+1) -- cycle; 
					\draw[fill=blue!30] (\x,\y+1,\z) -- (\x+1,\y+1,\z) -- (\x+1,\y+1,\z+1) -- (\x,\y+1,\z+1) -- cycle; 
					\draw[fill=blue!30] (\x,\y,\z+1) -- (\x+1,\y,\z+1) -- (\x+1,\y+1,\z+1) -- (\x,\y+1,\z+1) -- cycle; 
				}
			}
		}
		
		\foreach \x in {0,1,2,3}
		{
			\foreach \y in {1}
			{
				\foreach \z in {0}
				{
					\draw[fill=blue!30] (\x,\y,\z) -- (\x+1,\y,\z) -- (\x+1,\y+1,\z) -- (\x,\y+1,\z) -- cycle; 
					\draw[fill=blue!30] (\x,\y,\z) -- (\x,\y+1,\z) -- (\x,\y+1,\z+1) -- (\x,\y,\z+1) -- cycle; 
					\draw[fill=blue!30] (\x,\y,\z) -- (\x+1,\y,\z) -- (\x+1,\y,\z+1) -- (\x,\y,\z+1) -- cycle; 
					\draw[fill=blue!30] (\x+1,\y,\z) -- (\x+1,\y+1,\z) -- (\x+1,\y+1,\z+1) -- (\x+1,\y,\z+1) -- cycle; 
					\draw[fill=blue!30] (\x,\y+1,\z) -- (\x+1,\y+1,\z) -- (\x+1,\y+1,\z+1) -- (\x,\y+1,\z+1) -- cycle; 
					\draw[fill=blue!70] (\x,\y,\z+1) -- (\x+1,\y,\z+1) -- (\x+1,\y+1,\z+1) -- (\x,\y+1,\z+1) -- cycle; 
				}
			}
		}
		
		\foreach \x in {0,1}
		{
			\foreach \y in {2,3}
			{
				\foreach \z in {0}
				{
					\draw[fill=blue!30] (\x,\y,\z) -- (\x+1,\y,\z) -- (\x+1,\y+1,\z) -- (\x,\y+1,\z) -- cycle; 
					\draw[fill=blue!30] (\x,\y,\z) -- (\x,\y+1,\z) -- (\x,\y+1,\z+1) -- (\x,\y,\z+1) -- cycle; 
					\draw[fill=blue!30] (\x,\y,\z) -- (\x+1,\y,\z) -- (\x+1,\y,\z+1) -- (\x,\y,\z+1) -- cycle; 
					\draw[fill=blue!30] (\x+1,\y,\z) -- (\x+1,\y+1,\z) -- (\x+1,\y+1,\z+1) -- (\x+1,\y,\z+1) -- cycle; 
					\draw[fill=blue!30] (\x,\y+1,\z) -- (\x+1,\y+1,\z) -- (\x+1,\y+1,\z+1) -- (\x,\y+1,\z+1) -- cycle; 
					\draw[fill=blue!30] (\x,\y,\z+1) -- (\x+1,\y,\z+1) -- (\x+1,\y+1,\z+1) -- (\x,\y+1,\z+1) -- cycle; 
				}
			}
		}		
		\foreach \x in {0,1,...,4}
		{
			\foreach \y in {0,1,...,3}
			{
				\foreach \z in {3}
				{
					\node at (\x,\y,\z) {};
				}
			}
		}
		
		\foreach \x in {0,1,2,3}
		{
			\foreach \y in {0,1}
			{
				\foreach \z in {1}
				{
					\draw[fill=blue!30] (\x,\y,\z) -- (\x+1,\y,\z) -- (\x+1,\y+1,\z) -- (\x,\y+1,\z) -- cycle; 
					\draw[fill=blue!30] (\x,\y,\z) -- (\x,\y+1,\z) -- (\x,\y+1,\z+1) -- (\x,\y,\z+1) -- cycle; 
					\draw[fill=blue!30] (\x,\y,\z) -- (\x+1,\y,\z) -- (\x+1,\y,\z+1) -- (\x,\y,\z+1) -- cycle; 
					\draw[fill=blue!30] (\x+1,\y,\z) -- (\x+1,\y+1,\z) -- (\x+1,\y+1,\z+1) -- (\x+1,\y,\z+1) -- cycle; 
					\draw[fill=blue!30] (\x,\y+1,\z) -- (\x+1,\y+1,\z) -- (\x+1,\y+1,\z+1) -- (\x,\y+1,\z+1) -- cycle; 
					\draw[fill=blue!30] (\x,\y,\z+1) -- (\x+1,\y,\z+1) -- (\x+1,\y+1,\z+1) -- (\x,\y+1,\z+1) -- cycle; 
				}
			}
		}
		\foreach \x in {0,1}
		{
			\foreach \y in {2}
			{
				\foreach \z in {1}
				{
					\draw[fill=blue!30] (\x,\y,\z) -- (\x+1,\y,\z) -- (\x+1,\y+1,\z) -- (\x,\y+1,\z) -- cycle; 
					\draw[fill=blue!30] (\x,\y,\z) -- (\x,\y+1,\z) -- (\x,\y+1,\z+1) -- (\x,\y,\z+1) -- cycle; 
					\draw[fill=blue!30] (\x,\y,\z) -- (\x+1,\y,\z) -- (\x+1,\y,\z+1) -- (\x,\y,\z+1) -- cycle; 
					\draw[fill=blue!30] (\x+1,\y,\z) -- (\x+1,\y+1,\z) -- (\x+1,\y+1,\z+1) -- (\x+1,\y,\z+1) -- cycle; 
					\draw[fill=blue!30] (\x,\y+1,\z) -- (\x+1,\y+1,\z) -- (\x+1,\y+1,\z+1) -- (\x,\y+1,\z+1) -- cycle; 
					\draw[fill=blue!30] (\x,\y,\z+1) -- (\x+1,\y,\z+1) -- (\x+1,\y+1,\z+1) -- (\x,\y+1,\z+1) -- cycle; 
				}
			}
		}		
		\foreach \x in {1,3}
		{
			\foreach \y in {1,2}
			{
				\foreach \z in {3}
				{
					\node at (\x,\y,\z) {};
				}
			}
		}
		
		\foreach \x in {0,1,2}
		{
			\foreach \y in {0}
			{
				\foreach \z in {2}
				{
					\draw[fill=blue!30] (\x,\y,\z) -- (\x+1,\y,\z) -- (\x+1,\y+1,\z) -- (\x,\y+1,\z) -- cycle; 
					\draw[fill=blue!30] (\x,\y,\z) -- (\x,\y+1,\z) -- (\x,\y+1,\z+1) -- (\x,\y,\z+1) -- cycle; 
					\draw[fill=blue!30] (\x,\y,\z) -- (\x+1,\y,\z) -- (\x+1,\y,\z+1) -- (\x,\y,\z+1) -- cycle; 
					\draw[fill=blue!30] (\x+1,\y,\z) -- (\x+1,\y+1,\z) -- (\x+1,\y+1,\z+1) -- (\x+1,\y,\z+1) -- cycle; 
					\draw[fill=blue!30] (\x,\y+1,\z) -- (\x+1,\y+1,\z) -- (\x+1,\y+1,\z+1) -- (\x,\y+1,\z+1) -- cycle; 
					\draw[fill=blue!30] (\x,\y,\z+1) -- (\x+1,\y,\z+1) -- (\x+1,\y+1,\z+1) -- (\x,\y+1,\z+1) -- cycle; 
				}
			}
		}
		
		\foreach \x in {0,1}
		{
			\foreach \y in {1}
			{
				\foreach \z in {2}
				{
					\draw[fill=blue!30] (\x,\y,\z) -- (\x+1,\y,\z) -- (\x+1,\y+1,\z) -- (\x,\y+1,\z) -- cycle; 
					\draw[fill=blue!30] (\x,\y,\z) -- (\x,\y+1,\z) -- (\x,\y+1,\z+1) -- (\x,\y,\z+1) -- cycle; 
					\draw[fill=blue!30] (\x,\y,\z) -- (\x+1,\y,\z) -- (\x+1,\y,\z+1) -- (\x,\y,\z+1) -- cycle; 
					\draw[fill=blue!30] (\x+1,\y,\z) -- (\x+1,\y+1,\z) -- (\x+1,\y+1,\z+1) -- (\x+1,\y,\z+1) -- cycle; 
					\draw[fill=blue!30] (\x,\y+1,\z) -- (\x+1,\y+1,\z) -- (\x+1,\y+1,\z+1) -- (\x,\y+1,\z+1) -- cycle; 
					\draw[fill=blue!30] (\x,\y,\z+1) -- (\x+1,\y,\z+1) -- (\x+1,\y+1,\z+1) -- (\x,\y+1,\z+1) -- cycle; 
					 face
				}
			}
		}		
		
		\foreach \x in {0}
		{
			\foreach \y in {0}
			{
				\foreach \z in {3}
				{
					\node at (\x,\y,\z) {};
				}
			}
		}
		
		\foreach \x in {0}
		{
			\foreach \y in {0}
			{
				\foreach \z in {3}
				{
					\draw[fill=blue!30] (\x,\y,\z) -- (\x+1,\y,\z) -- (\x+1,\y+1,\z) -- (\x,\y+1,\z) -- cycle; 
					\draw[fill=blue!30] (\x,\y,\z) -- (\x,\y+1,\z) -- (\x,\y+1,\z+1) -- (\x,\y,\z+1) -- cycle; 
					\draw[fill=blue!30] (\x,\y,\z) -- (\x+1,\y,\z) -- (\x+1,\y,\z+1) -- (\x,\y,\z+1) -- cycle; 
					\draw[fill=blue!30] (\x+1,\y,\z) -- (\x+1,\y+1,\z) -- (\x+1,\y+1,\z+1) -- (\x+1,\y,\z+1) -- cycle; 
					\draw[fill=blue!30] (\x,\y+1,\z) -- (\x+1,\y+1,\z) -- (\x+1,\y+1,\z+1) -- (\x,\y+1,\z+1) -- cycle; 
					\draw[fill=blue!30] (\x,\y,\z+1) -- (\x+1,\y,\z+1) -- (\x+1,\y+1,\z+1) -- (\x,\y+1,\z+1) -- cycle; 
				}
			}
		}
		
		\foreach \x in {0,1}
		{
			\foreach \y in {1}
			{
				\foreach \z in {2}
				{
					\draw[fill=blue!30] (\x,\y,\z) -- (\x+1,\y,\z) -- (\x+1,\y+1,\z) -- (\x,\y+1,\z) -- cycle; 
					\draw[fill=blue!30] (\x,\y,\z) -- (\x,\y+1,\z) -- (\x,\y+1,\z+1) -- (\x,\y,\z+1) -- cycle; 
					\draw[fill=blue!30] (\x,\y,\z) -- (\x+1,\y,\z) -- (\x+1,\y,\z+1) -- (\x,\y,\z+1) -- cycle; 
					\draw[fill=blue!30] (\x+1,\y,\z) -- (\x+1,\y+1,\z) -- (\x+1,\y+1,\z+1) -- (\x+1,\y,\z+1) -- cycle; 
					\draw[fill=blue!30] (\x,\y+1,\z) -- (\x+1,\y+1,\z) -- (\x+1,\y+1,\z+1) -- (\x,\y+1,\z+1) -- cycle; 
					\draw[fill=blue!30] (\x,\y,\z+1) -- (\x+1,\y,\z+1) -- (\x+1,\y+1,\z+1) -- (\x,\y+1,\z+1) -- cycle; 
				}
			}
		}	
	\end{tikzpicture}
	&  
\end{eqnarray*}
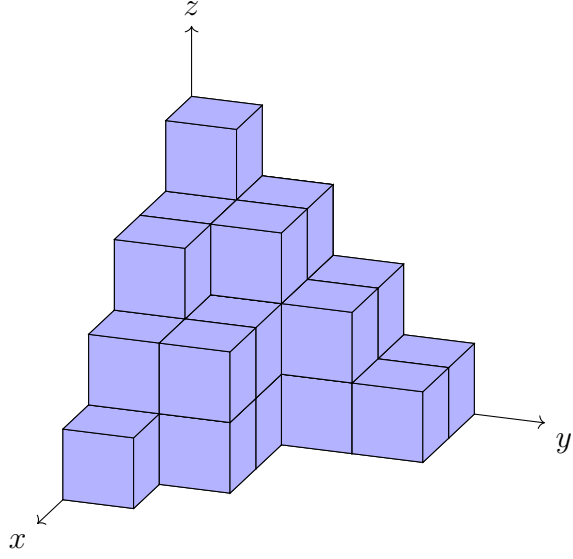
\centerline{\footnotesize \text{\bf Figure 4. } \text{$D_\la$ for a 
		3-dimensional partition of $ n = 29 $.}}
\end{example}

\begin{definition}  \label{HDbarcode}
	\begin{enumerate}
		\item[{\rm (i)}]
		A {\it multiset} is defined to be an ordered pair $(A, m)$ where 
		$A$ is a set and $m: A \to \Z^+$ is a function from $A$ to 
		the set of positive integers. Informally, a multiset is a set 
		where elements are allowed to have multiple copies. 
		
		\item[{\rm (ii)}] 
		For $d \ge 1$, a {\it $d$-dimensional barcode} is a multiset consisting of  
		$$
		\T_{a^{(i)}} (D^{\rm ei}_{\la^{(i)}}), \quad i \in \Lambda
		$$
		where $\Lambda$ is an index set, and for each $i \in \Lambda$, 
		$a^{(i)} \in \N^{d}$ and $\la^{(i)}$ is a $d$-dimensional partition.
	\end{enumerate}
\end{definition}

\section{\bf $\N^d$-indexed persistence modules}   
\label{section_MPP}

\begin{definition}   \label{PosetCat-def}
	\begin{enumerate}
		\item[{\rm (i)}] A {\it partially ordered set} (or {\it poset} for short) 
		is a set $P$ together with a partial ordering $\le$ satisfying
		\smallskip
		\begin{enumerate}
			\item[$\bullet$] $a \le a$ for all $a \in P$ (reflexivity);
			
			\item[$\bullet$] $a \le b$ and $b \le a$ imply $a=b$ (anti-symmetry);
			
			\item[$\bullet$]  $a \le b$ and $b \le c$ imply $a \le c$ 
			(transitivity).
		\end{enumerate}
		
		\item[{\rm (ii)}] The {\it poset category} associated to a poset $P$ 
		is the category whose objects are the elements of $P$ and for $a,b \in P$, 
		${\rm Hom}(a, b)$ consists of one element if $a \le b$ and 
		is the empty set if $a \not \le b$. By abusing notations, we also 
		use $P$ to denotes the poset category associated to a poset $P$. 
	\end{enumerate}
\end{definition}

\begin{definition}   \label{MPP-def}
	\begin{enumerate}
		\item[{\rm (i)}] 
		Fix a field $k$ and a poset $(P, \le)$.
		A {\it $P$-indexed persistence module over $k$} is a functor 
		$$
		M: P \to {\bf Vec}_k
		$$ 
		where ${\bf Vec}_k$ is the category of finite dimensional vector spaces 
		over $k$ (and $P$ denotes the poset category associated to $P$). 
		
		\item[{\rm (ii)}] 
		The image of $p \in P$ is denoted 
		by $M_p$. For $p, q \in P$ with $p \le q$, the unique morphism from 
		$M_p$ to $M_q$ corresponding to $p \le q$ is denoted by $M_{p,q}$.
		
		\item[{\rm (iii)}] 
		Define $\gr (m) = p$ if $m \in M_p$.
	\end{enumerate}
\end{definition}

In the rest of the paper, we will assume $P = \N^d$ and the field $k$ 
will be implicite. A persistence module is meant to be 
a $\N^d$-indexed persistence module over $k$.

\begin{theorem}  \label{CZ2009}
	(\cite[Theorem~1]{CZ}) (Correspondence) 
	The category of $\N^d$-indexed persistence modules over $k$ is equivalent  
	to the category of $\N^d$-graded $k[t_1, \ldots, t_d]$-modules.
\end{theorem}

\begin{theorem}  \label{CZ2009Theorem2}
	(\cite[Theorem~2]{CZ}) (Realization) 
	Let $k = \mathbb F_p$ for some prime $p$, let $q$ be a positive integer,
	and let $M$ be an $\N^d$-graded $k[x_1, \ldots, x_d]$-module. 
	Then there is an $\N^d$-filtered finite simplicial complex $X$ so that 
	$H_q(X, k) \cong M$ as $\N^d$-persistence modules.
\end{theorem}

Denote the standard basis of $\R^d$ by
\begin{eqnarray}    \label{standardbasis}
	\{\e_1, \ldots, \e_d\}.
\end{eqnarray}
Given an $\N^d$-indexed persistence module $M: \N^d \to {\bf Vec}_k$, 
the associated $\N^d$-graded $k[t_1, \ldots, t_d]$-module is 
\begin{eqnarray}    \label{CZ2009.1}
	\bigoplus_{z \in \N^d} M_z
\end{eqnarray}
on which $k[t_1, \ldots, t_d]$ acts via $t_i \cdot m = M_{z, z+\e_i}(m)$ 
for every $m \in M_z$ and $1 \le i \le d$.

Morphisms between persistence modules, their kernels and images, 
submodules and quotient modules are defined component-wise 
in the usual way.  

\begin{definition}   \label{generators}
	Let $M$ be an $\N^d$-indexed persistence module. 
	\begin{enumerate}
		\item[{\rm (i)}] 
		Let $S \subset \cup_{z \in \N^d} M_z$ be a subset. 
		The submodule $\langle S \rangle$ of $M$ generated by $S$ is 
		the submodule such that for each $z \in \N^d$, 
		$\langle S \rangle_z$ consists of all the elements
		$$
		\sum_{i=1}^n c_i \cdot M_{\gr(s_i), z}(s_i)
		$$
		where $c_1, \ldots, c_n \in k$ and $s_1, \ldots, s_n \in S$ 
		with $\gr(s_i) \le z$ for each $i$. If $\gr(s) \not \le z$ 
		for every $s \in S$, then we put $\langle S \rangle_z = 0$.
		
		\item[{\rm (ii)}] 
		A subset $S \subset \cup_{z \in \N^d} M_z$ is a {\it set of generators} 
		for $M$ if $M = \langle S \rangle$.
		
		\item[{\rm (iii)}] 
		The persistence module $M$ is {\it finitely generated} if there
		exists a finite set of generators for $M$.
		
		\item[{\rm (iv)}]
		Fix $a \in \N^d$. The translation $\T_aM$ of $M$ by $a$ is 
		the persistence module given by
		$$
		(\T_aM)_x = \begin{cases} 
			M_{x-a}&\text{if $x \ge a$},  \\
			0&\text{otherwise},
		\end{cases}
		\qquad
		(\T_aM)_{x,y} = \begin{cases} 
			M_{x-a, y-a}&\text{if $a \le x \le y$},  \\
			0&\text{otherwise},
		\end{cases}
		$$
		for $x, y \in \N^d$ with $x \le y$. 
	\end{enumerate}
\end{definition}

For a $d$-dimensional partition $\la$,
recall from Definition~\ref{Youngdiagram}~(ii) the set $D_\la^{\rm int}$ of integral points inside the Young diagram $D_\la$.

\begin{definition}   \label{kla}
	\begin{enumerate}
		\item[{\rm (i)}] 
		For a $d$-dimensional partition $\la$, 
		define the $\N^d$-indexed persistence module $\kbf_\la$ by
		$$
		(\kbf_\la)_x 
		= \begin{cases}
			k &\text{if $x \in D_\la^{\rm int}$}, \\
			0 &\text{otherwise}, 
		\end{cases}  
		\qquad
		(\kbf_\la)_{x, y} 
		= \begin{cases}
			\id_k &\text{if $x,y \in D_\la^{\rm int}$ and $x \le y$}, \\
			0 &\text{otherwise},
		\end{cases}       
		$$
		where $x, y \in \N^d$ with $x \le y$. When $\la = (+\infty)_{i_1, \ldots, 
			i_{d-1}}$, we put $\kbf_\la = \kbf$. When $|\la| \ne 0$, 
		define ${\bf 1}_\la \in (\kbf_\la)_O = k$ to 
		be the multiplicative identity. 
		
		\item[{\rm (ii)}] 
		An $\N^d$-indexed persistence module $F$ is free if there exists a multiset $A$ of elements in $\N^d$ such that 
		$$
		F \cong \bigoplus_{a \in A} \T_a\kbf.
		$$
		
		\item[{\rm (iii)}]
		A {\it presentation} of a persistence module $M$ is a morphism 
		$f: F \to F'$ of free modules such that $M \cong F'/\im(f)$.
	\end{enumerate}
\end{definition}

The following combines \cite[Proposition~6.40]{Les} and \cite[Proposition~6.43]{Les}.

\begin{proposition}   \label{Les6.40}
	Every persistence module $M$ has a presentation. Moreover, 
	if $M$ is finitely generated, then there exists a presentation 
	$f: F \to F'$ of $M$ such that both the free modules $F$ and $F'$ 
	are finitely generated. 
\end{proposition}

Next, we study the case $F' = \T_a\kbf$ where $a \in \N^d$.

\begin{lemma}  \label{quotientsofk}
	Let $a \in \N^d$. Then every quotient of the persistence module 
	$\T_a\kbf$ is equal to $\T_a\kbf_{\la}$ for 
	some $d$-dimensional partition $\la$.
\end{lemma}
\begin{proof}
	It suffices to prove that every quotient of the persistence module 
	$\kbf$ is equal to $\kbf_{\la}$ for some $d$-dimensional partition $\la$. 
	Let $\mathcal Q$ be a quotient of $\kbf$. 
	Let $\mathcal I$ be the submodule of $\kbf$ such that 
	$\mathcal Q = \kbf/\mathcal I$.
	By \eqref{CZ2009.1}, the persistence module $\kbf$ corresponds to 
	the $\N^d$-graded $k[t_1, \ldots, t_d]$-module 
	$k[t_1, \ldots, t_d]$ itself. Thus, the submodule $\mathcal I$ of 
	$\kbf$ corresponds to an $\N^d$-graded ideal $I$ of $k[t_1, \ldots, t_d]$. 
	Being $\N^d$-graded, the ideal $I$ of $k[t_1, \ldots, t_d]$ is generated 
	by monomials. Define a $d$-dimensional partition 
	$\la = (\la_{i_1, \ldots, i_{d-1}})_{i_1, \ldots, i_{d-1}}$ as follows: 
	$$
	\la_{i_1, \ldots, i_{d-1}}
	= \begin{cases}
		\min \{b|t_1^{i_1} \cdots t_{d-1}^{i_{d-1}}t_d^b \in I\}  
		&\text{if $t_1^{i_1} \cdots t_{d-1}^{i_{d-1}}t_d^b \in I$ 
			for some $b \in \N$}, \\
		+\infty &\text{otherwise}.
	\end{cases}
	$$
	Indeed, note that 
	if $t_1^{i_1} \cdots t_d^{i_d} \in I$, then $t_1^{j_1} \cdots t_d^{j_d} 
	\in I$ whenever $i_1 \le j_1, \ldots, i_d \le j_d$. 
	It follows that 
	$
	\lambda_{i_1, \ldots, i_{d-1}} \ge \lambda_{j_1, \ldots, j_{d-1}}
	$
	whenever $i_1 \le j_1, \ldots, i_{d-1} \le j_{d-1}$.
	
	We have
	\begin{eqnarray}    \label{quotientsofk.1}  
		I 
		= \bigoplus_{\substack{(i_1, \ldots, i_{d-1}) \in \N^{d-1}\\
				\la_{i_1, \ldots, i_{d-1}} \ne +\infty}} \,\,
		\bigoplus_{b \ge \la_{i_1, \ldots, i_{d-1}}} k \cdot 
		t_1^{i_1} \cdots t_{d-1}^{i_{d-1}} t_d^b.
	\end{eqnarray}
	As a vector space, the quotient $k[t_1, \ldots, t_d]/I$ is equal to 
	\begin{eqnarray}    \label{quotientsofk.2}  
		\bigoplus_{(i_1, \ldots, i_{d-1}) \in \N^{d-1}} \,\,
		\bigoplus_{b < \la_{i_1, \ldots, i_{d-1}}} k \cdot 
		t_1^{i_1} \cdots t_{d-1}^{i_{d-1}} t_d^b.
	\end{eqnarray}
	Via the correspondence \eqref{CZ2009.1}, the quotient $\mathcal Q$ of 
	$\kbf$, which corresponds to the quotient $k[t_1, \ldots, t_d]/I$ of 
	$k[t_1, \ldots, t_d]$, is equal to $\kbf_{\la}$. 
\end{proof}

\begin{definition}   \label{goodbarcode}
	An $\N^d$-indexed persistence module $M$ admits a {\it barcode} 
	if 
	\begin{eqnarray}    \label{goodbarcode.0} 
		M \cong \bigoplus_{i \in \Lambda} \T_{a^{(i)}}\kbf_{\la^{(i)}}
	\end{eqnarray}
	where for each $i \in \Lambda$, $a^{(i)} \in \N^{d}$ and 
	$\la^{(i)}$ is a $d$-dimensional partition with $|\la^{(i)}| \ne 0$. 
	In this case, the {\it barcode} ${\mathfrak B}_M$ of $M$ is defined 
	to be the multiset whose elements are
	$$
	\T_{a^{(i)}} (D^{\rm ei}_{\la^{(i)}}), \quad i \in \Lambda.
	$$
\end{definition}

\begin{lemma}   \label{Lemma_GoodBarcode}
	Assume that an $\N^d$-indexed persistence module $M$ admits 
	a barcode ${\mathfrak B}_M$. Then,
	\begin{eqnarray}\label{Lemma_GoodBarcode.0} 
		\text{\rm Rank}(M_{x,y}) = |\{S \in {\mathfrak B}_M| x,y \in S \}|
	\end{eqnarray}
	for all $x, y \in \N^d$ satisfying $x \le y$.
\end{lemma}
\par\noindent
{\it Proof.}
We may assume $M = \bigoplus_{i \in \Lambda} 
\T_{a^{(i)}}\kbf_{\la^{(i)}}$ so that 
$$
  {\mathfrak B}_M 
= \big \{ \T_{a^{(i)}} (D^{\rm ei}_{\la^{(i)}}) \big \}_{i \in \Lambda}
$$
as multisets. Let $x, y \in \N^d$ satisfying $x \le y$. Then, 
\begin{eqnarray}    \label{Lemma_GoodBarcode.1} 
  \text{\rm Rank}(M_{x,y}) 
= \sum_{i \in \Lambda} \text{\rm Rank}
  \big ((\T_{a^{(i)}}\kbf_{\la^{(i)}})_{x,y} \big ).
\end{eqnarray}
		
Put 
$$
\Lambda_1 
= \{i \in \Lambda|x,y \in \T_{a^{(i)}} (D^{\rm int}_{\la^{(i)}})\}.
$$
By Definition~\ref{kla}~(i) and Definition~\ref{generators}~(iv), 
$$
\text{\rm Rank} \big ((\T_{a^{(i)}}\kbf_{\la^{(i)}})_{x,y} \big )
= \begin{cases}
	1, &\text{\rm if $x, y \in \T_{a^{(i)}} (D^{\rm int}_{\la^{(i)}})$} \\
	0, &\text{\rm otherwise.}
  \end{cases}
$$
Therefore, $\text{\rm Rank}(M_{x,y}) = |\Lambda_1| 
= |\{ \T_{a^{(i)}} (D^{\rm int}_{\la^{(i)}}) | 
x,y \in \T_{a^{(i)}} (D^{\rm int}_{\la^{(i)}}) \}|$.
By Remark~\ref{Dlaint-ei},
$$
D_\la^{\rm int} = D_\la^{\rm ei} \cap \N^d
$$
for every $d$-dimensional partition $\la$. Hence, 
\begin{equation}
  \text{\rm Rank}(M_{x,y}) 
= |\{ \T_{a^{(i)}} (D^{\rm ei}_{\la^{(i)}}) \in {\mathfrak B}_M| 
  x,y \in \T_{a^{(i)}} (D^{\rm ei}_{\la^{(i)}}) \}|.
\tag*{$\qed$}
\end{equation}

The following theorem has been proved in \cite{Nat}.

\begin{theorem}   \label{UniqueBarcode}
(\cite{Nat}) Let $M$ be an $\N^d$-indexed persistence module. 
Assume that there are two isomorphisms 
$$
M \cong \bigoplus_{i \in \Lambda_1} \T_{a^{(i)}}\kbf_{\la^{(i)}}
\qquad
\text{and}
\qquad
M \cong \bigoplus_{\ell \in \Lambda_2} \T_{b^{(\ell)}}\kbf_{\mu^{(\ell)}}
$$
with $|\la^{(i)}| \ne 0$ and $|\mu^{(\ell)}| \ne 0$ for all 
$i \in \Lambda_1$ and $\ell \in \Lambda_2$. Then as multisets, 
\begin{eqnarray}   \label{UniqueBarcode.0}
  \{(\la^{(i)}, a^{(i)})\}_{i \in \Lambda_1}
= \{(\mu^{(\ell)}, b^{(\ell)})\}_{\ell \in \Lambda_2}.
\end{eqnarray}
In particular, the barcode ${\mathfrak B}_M$ of $M$ in 
Definition~\ref{goodbarcode} is well-defined.
\end{theorem}

Lemma~\ref{Lemma_GoodBarcode} indicates that if 
an $\N^d$-indexed persistence module $M$ admits 
a barcode ${\mathfrak B}_M$, then ${\mathfrak B}_M$ is a {\it good} barcode 
in the sense of \cite[Definition~10.11]{Les}; moreover, 
by Theorem~\ref{UniqueBarcode}, 
the barcode ${\mathfrak B}_M$ of $M$ is well-defined. 

\begin{example}   \label{example_reviewer2}
Consider the $\N^2$-indexed persistence module $M$ defined by 
$$
M = \kbf_{\la^{(1)}} \oplus \kbf_{\la^{(2)}} \oplus \kbf_{\la^{(3)}}
$$
where $\la^{(1)} = (3, 2, 2)$, $\la^{(2)} = (2, 1)$ and $\la^{(3)} = (3, 3)$ 
are $2$-dimensional partitions. By Definition~\ref{goodbarcode}, 
the barcode ${\mathfrak B}_M$ of $M$ is the multiset consisting of 
$D_{\la^{(1)}}^{\rm ei}, D_{\la^{(2)}}^{\rm ei}, D_{\la^{(3)}}^{\rm ei}$
which are illustrated in Figure~5.	
	\begin{center}
		\begin{tikzpicture}[x=0.5cm,y=0.5cm,step=0.5cm]
			\fill[blue!30] (0,3) -- (1,3) -- (1,2) -- (2,2) -- (3,2) -- (3,0) -- (0,0) -- cycle;
			\draw[->] (0,0) -- (4,0);
			\draw[->] (0,0) -- (0,4);
			=========================
			\draw[dashed] (0,3) -- (1,3);
			\draw[dashed] (1,3) -- (1,2);
			\draw[dashed] (1,2) -- (3,2);
			\draw[dashed] (3,2) -- (3,0);
			\draw[dashed] (3,1) -- (3,0);
			\filldraw[fill=white] (0,3) circle (2pt);
			\filldraw[fill=white] (3,0) circle (2pt);
			
			\node at (-0.4,-0.4) {0};
			\foreach \x in {1,2}
			{
				\node[fill,circle, inner sep=1pt,label=below:$\x$] at (\x,0) {};
				\node[fill, circle, inner sep=1pt,label=left:$\x$] at (0,\x) {};
			}
			\node[circle, inner sep=1pt, label=left:3] at (0,3) {};
			\node[circle, inner sep=1pt, label=below:3] at (3,0) {};
		\end{tikzpicture}
		\begin{tikzpicture}[x=0.5cm,y=0.5cm,step=0.5cm]
			\fill[blue!30] (0,2) -- (1,2) -- (1,1) -- (2,1) -- (2,0) -- (0,0) -- cycle;
			\draw[->] (0,0) -- (3,0);
			\draw[->] (0,0) -- (0,3);
			=========================
			\draw[dashed] (0,2) -- (1,2);
			\draw[dashed] (1,2) -- (1,1);
			\draw[dashed] (1,1) -- (2,1);
			\draw[dashed] (2,1) -- (2,0);
			\filldraw[fill=white] (0,2) circle (2pt);
			\filldraw[fill=white] (2,0) circle (2pt);
			
			\node at (-0.4,-0.4) {0};
			\foreach \x in {1}
			{
				\node[fill,circle, inner sep=1pt,label=below:$\x$] at (\x,0) {};
				\node[fill, circle, inner sep=1pt,label=left:$\x$] at (0,\x) {};
			}
			\node[circle, inner sep=1pt, label=left:2] at (0,2) {};
			\node[circle, inner sep=1pt, label=below:2] at (2,0) {};
		\end{tikzpicture}
		\begin{tikzpicture}[x=0.5cm,y=0.5cm,step=0.5cm]
			\fill[blue!30] (0,3) -- (2,3) -- (2,0)-- (0,0) -- cycle;
			\draw[->] (0,0) -- (4,0);
			\draw[->] (0,0) -- (0,4);
			=========================
			\draw[dashed] (0,3) -- (2,3);
			\draw[dashed] (2,3) -- (2,0);
			\filldraw[fill=white] (0,3) circle (2pt);
			\filldraw[fill=white] (2,0) circle (2pt);
			
			\node at (-0.4,-0.4) {0};
			\foreach \x in {1,2}
			{
				\node[fill, circle, inner sep=1pt,label=left:$\x$] at (0,\x) {};
			}
			
			\foreach \x in {1}
			{
				\node[fill, circle, inner sep=1pt,label=below:$\x$] at (\x, 0) {};
			}
			\node[circle, inner sep=1pt, label=left:3] at (0,3) {};
			\node[circle, inner sep=1pt, label=below:2] at (2,0) {};
		\end{tikzpicture}
	\end{center}
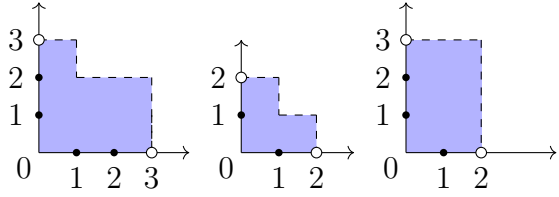
	\centerline{\footnotesize \text{\bf Figure 5. } From left to right, \text{$D_\la^{\rm ei}$ for 
			$ \la^{(1)} = (3, 2, 2) $, $ \la^{(2)} = (2, 1) $ and $ \la^{(3)} = (3, 3) $.}}
\par\noindent
Note that the three Young diagrams $D_{\la^{(1)}}, D_{\la^{(2)}}, D_{\la^{(3)}}$ 
are the three shaded regions in Figure~5 together with 
their respective dashed boundaries. 
\end{example}

\section{\bf The rank invariant}   
\label{section_rank}

In this section, we study the relation between the rank invariant 
which is defined in \cite{CZ} 
and the $\N^d$-indexed persistence modules which admit barcodes. 
We will prove a necessary and sufficient condition for 
two $\N^d$-indexed persistence modules admitting barcodes to have 
the same rank invariant. 

The following definitions are adopted from \cite{CZ} (see also \cite{Les}). 

\begin{definition}  \label{rk-invariant}
\begin{enumerate}
\item[{\rm (i)}] 
Given a poset $P$, define
$$
P^{\le} = \{(x,y) \in P \times P|x \le y\}.
$$

\item[{\rm (ii)}] 
The {\it rank invariant} of an $\N^d$-indexed persistence module $M$ is 
the function ${\rm Rank}^M: (\N^d)^{\le} \to \N$ given by
$$
{\rm Rank}^M(x, y) = \text{\rm Rank}(M_{x,y}).
$$
\end{enumerate}
\end{definition}

\begin{lemma}   \label{Lemma2_GoodBarcode}
Assume that an $\N^d$-indexed persistence module $M$ admits 
a barcode ${\mathfrak B}_M$. Then,
\begin{eqnarray}    \label{Lemma2_GoodBarcode.0} 
{\rm Rank}^M(x, y) = |\{S \in {\mathfrak B}_M| x,y \in S \}|
\end{eqnarray}
for every $(x,y) \in (\N^d)^{\le}$.
\end{lemma}
\begin{proof}
Follows directly from Definition~\ref{rk-invariant}~(ii) and 
Lemma~\ref{Lemma_GoodBarcode}.
\end{proof}

\begin{lemma}  \label{rk-inv-MN}
Let $d \ge 1$. 
Let $M$ and $N$ be $\N^d$-indexed persistence modules admitting barcodes:
$$
M = \bigoplus_{i \in \Lambda_1} \kbf_{\la^{(i)}}
\qquad
\text{and}
\qquad
N = \bigoplus_{\ell \in \Lambda_2} \kbf_{\mu^{(\ell)}}
$$
with $|\la^{(i)}|, |\mu^{(\ell)}| \ne 0$. If ${\rm Rank}^M = {\rm Rank}^N$, 
then for every $(i_1, \ldots, i_{d-1}) \in \N^{d-1}$, the two multisets 
\begin{eqnarray}    \label{rk-inv-MN.01}
\big \{ \la^{(i)}_{i_1, \ldots, i_{d-1}}|i \in \Lambda_1 \text{ and } 
   \la^{(i)}_{i_1, \ldots, i_{d-1}} > 0\big \}
\end{eqnarray}
and 
\begin{eqnarray}    \label{rk-inv-MN.02}
\big \{ \mu^{(\ell)}_{i_1, \ldots, i_{d-1}}|\ell \in \Lambda_2 \text{ and } 
   \mu^{(\ell)}_{i_1, \ldots, i_{d-1}} > 0\big \}
\end{eqnarray}
are equal. 
\end{lemma}
\begin{proof}
Fix $(i_1, \ldots, i_{d-1}) \in \N^{d-1}$.
Let $a_1, \ldots, a_s$ be the distinct values in the multiset \eqref{rk-inv-MN.01} 
with $a_1 > \ldots > a_s$, and let $m_1, \ldots, m_s$ be the multiplicities of 
$a_1, \ldots, a_s$ respectively in the multiset \eqref{rk-inv-MN.01}. Similarly, 
let $b_1, \ldots, b_t$ be the distinct values in the multiset \eqref{rk-inv-MN.02} 
with $b_1 > \ldots > b_t$, and let $n_1, \ldots, n_t$ be the multiplicities of 
$b_1, \ldots, b_t$ respectively in the multiset \eqref{rk-inv-MN.02}. 

Without loss of generality, assume $a_1 \ge b_1$.
By \eqref{Lemma2_GoodBarcode.0},
$$
  {\rm Rank}^M\big ((i_1, \ldots, i_{d-1}, 0), (i_1, \ldots, i_{d-1}, h)\big )
= \begin{cases}
  m_1 &\text{if $a_2 \le h < a_1$}, \\
  0 &\text{if $h \ge a_1$}.
  \end{cases}  
$$
Since ${\rm Rank}^M = {\rm Rank}^N$, we conclude that 
$$
  {\rm Rank}^N\big ((i_1, \ldots, i_{d-1}, 0), (i_1, \ldots, i_{d-1}, h)\big )
= \begin{cases}
  m_1 &\text{if $a_2 \le h < a_1$}, \\
  0 &\text{if $h \ge a_1$}.
  \end{cases} 
$$
It follows that $b_1 = a_1$ and $n_1 = m_1$.  

Next, without loss of generality, assume $a_2 \ge b_2$. 
By \eqref{Lemma2_GoodBarcode.0} again,
$$
  {\rm Rank}^M\big ((i_1, \ldots, i_{d-1}, 0), (i_1, \ldots, i_{d-1}, h)\big )
= \begin{cases}
  m_1 + m_2 &\text{if $a_3 \le h < a_2$}, \\
  m_1 &\text{if $a_2 \le h < a_1$}, \\
  0 &\text{if $h \ge a_1$}.
  \end{cases}  
$$
Since ${\rm Rank}^M = {\rm Rank}^N$, we conclude that 
$$
  {\rm Rank}^N\big ((i_1, \ldots, i_{d-1}, 0), (i_1, \ldots, i_{d-1}, h)\big )
= \begin{cases}
  m_1 + m_2 &\text{if $a_3 \le h < a_2$}, \\
  m_1 &\text{if $a_2 \le h < a_1$}, \\
  0 &\text{if $h \ge a_1$}.
  \end{cases}  
$$
It follows that $b_2 = a_2$ and $n_2 = m_2$. 

Repeating the above arguments, we see that $s=t$, $a_i = b_i$ and 
$m_i = n_i$ for all $1 \le i \le s=t$. Therefore, the two multisets 
\eqref{rk-inv-MN.01} and \eqref{rk-inv-MN.02} are equal.
\end{proof}

Unfortunately, under the conditions in Lemma~\ref{rk-inv-MN}, 
the two multisets $\big \{ \la^{(i)}|i \in \Lambda_1 \big \}$
and $\big \{ \mu^{(\ell)}|\ell \in \Lambda_2 \big \}$
may not be the same. An example is given below. 

\begin{example}   \label{example_MN}
Let $M, N: \N^2 \to {\bf Vec}_k$ be persistence modules given by
$$
M = \kbf_{\la^{(1)}} \oplus \kbf_{\la^{(2)}}
\qquad
\text{and}
\qquad
N = \kbf_{\mu^{(1)}} \oplus \kbf_{\mu^{(2)}}
$$
where $\la^{(1)} = (2)$, $\la^{(2)} = (1^2)$, $\mu^{(1)} = (1)$, 
$\mu^{(2)} = (2,1)$ are $2$-dimensional partitions.
Intuitively, $M$ and $N$ can be illustrated by 
\[
M = \begin{tikzcd}
	0 \arrow[r] & 0  \\
	k \arrow[r] \arrow[u] & 0 \arrow[u]\\
	k \arrow[r] \arrow[u]  & 0 \arrow[u]
\end{tikzcd}
 \bigoplus 
 \begin{tikzcd}
 	0 \arrow[r] & 0	\arrow[r] & 0  \\
 	k \arrow[r] \arrow[u] & k \arrow[r] \arrow[u] 	&  0 \arrow[u]
 \end{tikzcd}
\]
and 
\[
N = \begin{tikzcd}
	0 \arrow[r] & 0  \\
	k \arrow[r] \arrow[u] & 0 \arrow[u] \\
\end{tikzcd}
\bigoplus 
\begin{tikzcd}
	0 \arrow[r] & 0	 	&	  \\
	k \arrow[r] \arrow[u] & 0  \arrow[r] \arrow[u] 	&   0 \\
	k \arrow[r] \arrow[u] & k \arrow[r] \arrow[u] 	&  0  \arrow[u]
\end{tikzcd}
\]
where all the nontrivial maps are the identity maps. Then 
$$
 {\rm Rank }^M(x,y) = {\rm Rank }^N(x,y) = \begin{cases}
 	1	&	\text{if } x = (0,0) \text{ and } y = (1,0), \\
	1	&	\text{if } x = (0,0) \text{ and } y = (0,1), \\
	0	&	\text{ otherwise }.
 \end{cases}
$$
So ${\rm Rank }^M = {\rm Rank }^N$ as asserted by Lemma~\ref{rk-inv-MN}. 
However, the two multisets $\{\la^{(1)}, \la^{(2)}\}$ and $\{\mu^{(1)}, \mu^{(2)}\}$ 
are not equal. By Theorem~\ref{UniqueBarcode}, 
the two $\N^2$-indexed persistence modules $M$ and $N$ are not isomorphic. 
\end{example}

Our next goal is to present a necessary and sufficient condition for 
two $\N^d$-indexed persistence modules admitting barcodes to have 
the same rank invariants. We start with the projections $\mathfrak p$ and 
$\mathfrak q$.

\begin{definition}   \label{pd-1qd}
For $d \ge 2$, we define $\mathfrak p: \N^d \to \N^{d-1}$ by 
$$
\mathfrak p(x_1, \ldots, x_d) = (x_1, \ldots, x_{d-1}),
$$
and define $\mathfrak q: \N^d \to \N$ by 
$$
\mathfrak q(x_1, \ldots, x_d) = x_d.
$$
For $ d = 1 $, we define $  \mathfrak p : \N \to \N^0 = \{O\} $ by $ \mathfrak p(x) = O $.   
\end{definition}

\begin{lemma}   \label{xyTaLa}
Let $a \in \N^d$ and $\la$ be a $d$-dimensional partition with $|\la| \ne 0$.
Let $(x, y) \in (\N^d)^{\le}$. Then, $x, y \in \T_{a} (D^{\rm ei}_{\la})$ 
if and only if $a \le x$ and $\mathfrak q(y-a) < \la_{\mathfrak p(y-a)}$.
\end{lemma}
\begin{proof}
Recall the maps $\mathfrak p$ and $\mathfrak q$ from Definition~\ref{pd-1qd}.
By \eqref{Youngdiagram.2}, $z \in D^{\rm ei}_{\la}$ if and only if
$O \le z$ and $\mathfrak q(z) < \la_{\mathfrak p(z)}$.
Since $(x, y) \in (\N^d)^{\le}$, we have $x \le y$. 
Thus, 
\begin{eqnarray*}
 x, y \in \T_{a}(D^{\rm ei}_{\la}) 
&\text{\rm if and only if}&x-a, y-a \in D^{\rm ei}_{\la},  \\
&\text{\rm if and only if}&a \le x \text{ and } y-a \in D^{\rm ei}_{\la}, \\
&\text{\rm if and only if}&a \le x \text{ and } 
    \mathfrak q(y-a) < \la_{\mathfrak p(y-a)}.
\end{eqnarray*}
This completes the proof of the lemma.
\end{proof}

The following is the main result of the paper.

\begin{theorem}  \label{ThmRk-inv-MN}
Let $d \ge 1$. 
Let $M$ and $N$ be $\N^d$-indexed persistence modules admitting barcodes:
$$
M = \bigoplus_{i \in \Lambda_1} \T_{a^{(i)}}\kbf_{\la^{(i)}}
\qquad
\text{and}
\qquad
N = \bigoplus_{\ell \in \Lambda_2} \T_{b^{(\ell)}}\kbf_{\mu^{(\ell)}}
$$
where $|\la^{(i)}| \ne 0$ and $|\mu^{(\ell)}| \ne 0$ for all 
$i \in \Lambda_1$ and $\ell \in \Lambda_2$. Then, ${\rm Rank}^M = {\rm Rank}^N$ 
if and only if for every $(i_1, \ldots, i_{d-1}) \in \N^{d-1}$, the two multisets 
\begin{eqnarray}    \label{ThmRk-inv-MN.01}
\Big \{\big (a^{(i)}, (\la^{(i)})_{i_1, \ldots, i_{d-1}} \big )
   |i \in \Lambda_1 \text{ and } (\la^{(i)})_{i_1, \ldots, i_{d-1}} > 0 \Big \}
\end{eqnarray}
and 
\begin{eqnarray}    \label{ThmRk-inv-MN.02}
\Big  \{ \big (b^{(\ell)}, (\mu^{(\ell)})_{i_1, \ldots, i_{d-1}} \big )
   |\ell \in \Lambda_2 \text{ and } 
   (\mu^{(\ell)})_{i_1, \ldots, i_{d-1}} > 0 \Big \}
\end{eqnarray}
are equal. 
\end{theorem}
\begin{proof}
First of all, assume that the two multisets \eqref{ThmRk-inv-MN.01} and 
\eqref{ThmRk-inv-MN.02} are equal. Fix $(x, y) \in (\N^d)^{\le}$. 
By Lemma~\ref{Lemma2_GoodBarcode}, 
$$
{\rm Rank}^M(x, y) = |\{S \in {\mathfrak B}_M| x,y \in S \}|.
$$
Since ${\mathfrak B}_M 
= \big \{ \T_{a^{(i)}} (D^{\rm ei}_{\la^{(i)}}) \big \}_{i \in \Lambda_1}$, 
we see from Lemma~\ref{xyTaLa} that ${\rm Rank}^M(x, y)$ is equal to 
the cardinality of the multiset
\begin{eqnarray}    \label{ThmRk-inv-MN.1}
\Big \{\big (a^{(i)}, (\la^{(i)})_{\mathfrak p(y-a^{(i)})} \big )
   |i \in \Lambda_1, a^{(i)} \le x \text{ and } 
   \mathfrak q(y-a^{(i)}) < (\la^{(i)})_{\mathfrak p(y-a^{(i)})} \Big \}.
\end{eqnarray}
Similarly, ${\rm Rank}^N(x, y)$ is equal to 
the cardinality of the multiset
\begin{eqnarray}    \label{ThmRk-inv-MN.2}
\Big \{\big (b^{(\ell)}, (\mu^{(\ell)})_{\mathfrak p(y-b^{(\ell)})} \big )
   |\ell \in \Lambda_2, b^{(\ell)} \le x \text{ and } 
   \mathfrak q(y-b^{(\ell)}) < (\mu^{(\ell)})_{\mathfrak p(y-b^{(\ell)})} \Big \}.
\end{eqnarray}
Since the two multisets \eqref{ThmRk-inv-MN.01} and 
\eqref{ThmRk-inv-MN.02} are equal, so are the two multisets 
\eqref{ThmRk-inv-MN.1} and \eqref{ThmRk-inv-MN.2}. 
Hence ${\rm Rank}^M(x, y) = {\rm Rank}^N(x, y)$ for every $(x, y) \in (\N^d)^{\le}$. 
It follows that ${\rm Rank}^M = {\rm Rank}^N$.

Conversely, assume that ${\rm Rank}^M = {\rm Rank}^N$. 
Without loss of generality, assume that $a^{(1)}$ is a minimal element in 
the multiset 
\begin{eqnarray}    \label{ThmRk-inv-MN.3}
\{a^{(i)}|i \in \Lambda_1 \} \cup \{b^{(\ell)}|\ell \in \Lambda_2 \}.
\end{eqnarray}
Then, ${\rm Rank}^N(a^{(1)}, a^{(1)}) = {\rm Rank}^M(a^{(1)}, a^{(1)}) > 0$. 
By Lemma~\ref{Lemma2_GoodBarcode}, $a^{(1)} \in {\mathfrak B}_N$.
Since ${\mathfrak B}_N$ is the multiset consisting of all 
$\T_{b^{(\ell)}}\big (D^{\rm ei}_{\mu^{(\ell)}} \big )$ with $\ell \in \Lambda_2$, 
we have $a^{(1)} \in \T_{b^{(\ell)}}\big (D^{\rm ei}_{\mu^{(\ell)}} \big )$ 
for some $\ell \in \Lambda_2$. By Lemma~\ref{xyTaLa}, $b^{(\ell)} \le a^{(1)}$ 
for some $\ell \in \Lambda_2$. Since $a^{(1)}$ is a minimal element in 
the multiset \eqref{ThmRk-inv-MN.3}, we must have $b^{(\ell)} = a^{(1)}$. 
Without loss of generality, we may let $\ell = 1$ so that $b^{(1)} = a^{(1)}$. 
Put 
$$
\Lambda_1' = \{i \in \Lambda_1|a^{(i)} = a^{(1)}\}, \quad
\W M = \bigoplus_{i \in \Lambda_1'} \kbf_{\la^{(i)}}, 
$$
$$
M' = \bigoplus_{i \in \Lambda_1'} \T_{a^{(i)}}\kbf_{\la^{(i)}}
   = \T_{a^{(1)}} \W M, \quad
M'' = \bigoplus_{i \in \Lambda_1-\Lambda_1'} \T_{a^{(i)}}\kbf_{\la^{(i)}}
$$
and 
$$
\Lambda_2' = \{\ell \in \Lambda_2|b^{(\ell)} = b^{(1)} = a^{(1)}\}, \quad
\W N = \bigoplus_{\ell \in \Lambda_2'} \kbf_{\mu^{(\ell)}}, 
$$
$$
N' = \bigoplus_{\ell \in \Lambda_2'} \T_{b^{(\ell)}}\kbf_{\mu^{(\ell)}}
   = \T_{a^{(1)}} \W N, \quad
N'' = \bigoplus_{\ell \in \Lambda_2-\Lambda_2'} \T_{b^{(\ell)}}\kbf_{\mu^{(\ell)}}.
$$
We have $M = M' \oplus M''$ and $N = N' \oplus N''$.

\bigskip\noindent
{\bf Claim.} {\it ${\rm Rank}^{\W M} = {\rm Rank}^{\W N}$ and 
${\rm Rank}^{M''} = {\rm Rank}^{N''}$}.
\begin{proof}
Let $(x, y) \in (\N^d)^{\le}$. Since $\W M$ is generated at the origin $O$, 
$$
{\rm Rank}^{\W M}(x, y) = {\rm Rank}^{\W M}(O, y) 
= {\rm Rank}^{\T_{a^{(1)}} \W M}(a^{(1)}, y+a^{(1)})
= {\rm Rank}^{M'}(a^{(1)}, y+a^{(1)}).
$$
Since $a^{(1)}$ is a minimal element in the multiset $\{a^{(i)}|i \in \Lambda_1 \}$,  
we have $M''_{a^{(1)}, y+a^{(1)}} = 0$ and ${\rm Rank}^{M'}(a^{(1)}, y+a^{(1)}) 
= {\rm Rank}^{M}(a^{(1)}, y+a^{(1)})$. Thus, 
\begin{eqnarray}    \label{ThmRk-inv-MN.4}
{\rm Rank}^{\W M}(x, y) = {\rm Rank}^{M}(a^{(1)}, y+a^{(1)}).
\end{eqnarray}
Similarly, ${\rm Rank}^{\W N}(x, y) = {\rm Rank}^{N}(a^{(1)}, y+a^{(1)})$. 
Combining with ${\rm Rank}^M = {\rm Rank}^N$ and \eqref{ThmRk-inv-MN.4},
we conclude that ${\rm Rank}^{\W M}(x, y) = {\rm Rank}^{\W N}(x, y)$ 
for every $(x, y) \in (\N^d)^{\le}$. Therefore, we obtain
\begin{eqnarray}    \label{ThmRk-inv-MN.5}
{\rm Rank}^{\W M} = {\rm Rank}^{\W N}. 
\end{eqnarray}

Next, we prove that ${\rm Rank}^{M''} = {\rm Rank}^{N''}$. We have 
$$
  {\rm Rank}^{M'}(x, y) = {\rm Rank}^{\T_{a^{(1)}} \W M}(x, y)
= \begin{cases}
  {\rm Rank}^{\W M}(O, y-a^{(1)})&\text{if $a^{(1)} \le x$}, \\
  0 &\text{otherwise}.
  \end{cases}  
$$
Similarly, 
$$
  {\rm Rank}^{N'}(x, y)  
= \begin{cases}
  {\rm Rank}^{\W N}(O, y-a^{(1)})&\text{if $a^{(1)} \le x$}, \\
  0 &\text{otherwise}.
  \end{cases}  
$$
By \eqref{ThmRk-inv-MN.5}, ${\rm Rank}^{M'}(x, y) = {\rm Rank}^{N'}(x, y)$. Since 
$$
{\rm Rank}^{M}(x, y) = {\rm Rank}^{M'}(x, y) + {\rm Rank}^{M''}(x, y),
$$
$$
{\rm Rank}^{N}(x, y) = {\rm Rank}^{N'}(x, y) + {\rm Rank}^{N''}(x, y)
$$
and ${\rm Rank}^{M}(x, y) = {\rm Rank}^{N}(x, y)$, we get 
${\rm Rank}^{M''}(x, y) = {\rm Rank}^{N''}(x, y)$ 
for every $(x, y) \in (\N^d)^{\le}$. Therefore, 
${\rm Rank}^{M''} = {\rm Rank}^{N''}$.
\end{proof}

We continue the proof of the theorem. Fix $(i_1, \ldots, i_{d-1}) \in \N^{d-1}$.  
Applying Lemma~\ref{rk-inv-MN} to $\W M$ and $\W N$ with 
${\rm Rank}^{\W M} = {\rm Rank}^{\W N}$, we see that the two multisets
$$
\big \{(\la^{(i)})_{i_1, \ldots, i_{d-1}} 
|i \in \Lambda_1' \text{ and } (\la^{(i)})_{i_1, \ldots, i_{d-1}} > 0 \big \}
$$
and 
$$
\big \{(\mu^{(\ell)})_{i_1, \ldots, i_{d-1}} 
|\ell \in \Lambda_2' \text{ and } (\mu^{(\ell)})_{i_1, \ldots, i_{d-1}} > 0 \big \}
$$
are equal, i.e., the two multisets
\begin{eqnarray}    \label{ThmRk-inv-MN.6}
\Big \{\big (a^{(i)}, (\la^{(i)})_{i_1, \ldots, i_{d-1}} \big )
|i \in \Lambda_1' \text{ and } (\la^{(i)})_{i_1, \ldots, i_{d-1}} > 0 \Big \}
\end{eqnarray}
and 
\begin{eqnarray}    \label{ThmRk-inv-MN.7}
\Big \{\big (b^{(\ell)}, (\mu^{(\ell)})_{i_1, \ldots, i_{d-1}} \big )
|\ell \in \Lambda_2' \text{ and } (\mu^{(\ell)})_{i_1, \ldots, i_{d-1}} > 0 \Big \}
\end{eqnarray}
are equal. Applying induction to $M''$ and $N''$ with ${\rm Rank}^{M''} 
= {\rm Rank}^{N''}$, we conclude that 
\begin{eqnarray}    \label{ThmRk-inv-MN.8}
\Big \{\big (a^{(i)}, (\la^{(i)})_{i_1, \ldots, i_{d-1}} \big )
|i \in \Lambda_1-\Lambda_1' \text{ and } 
(\la^{(i)})_{i_1, \ldots, i_{d-1}} > 0 \Big \}
\end{eqnarray}
and 
\begin{eqnarray}    \label{ThmRk-inv-MN.9}
\Big  \{ \big (b^{(\ell)}, (\mu^{(\ell)})_{i_1, \ldots, i_{d-1}} \big )
|\ell \in \Lambda_2-\Lambda_2' \text{ and } 
(\mu^{(\ell)})_{i_1, \ldots, i_{d-1}} > 0 \Big \}
\end{eqnarray}
are equal. Combining \eqref{ThmRk-inv-MN.6}-\eqref{ThmRk-inv-MN.9}, 
we see that the two multisets \eqref{ThmRk-inv-MN.01} and 
\eqref{ThmRk-inv-MN.02} are equal.
\end{proof}

\begin{corollary}  \label{CorRk-inv-MN}
Let $d \ge 1$. 
Let $M$ and $N$ be $\N^d$-indexed persistence modules admitting barcodes:
$$
M = \bigoplus_{i \in \Lambda_1} \T_{a^{(i)}}\kbf_{\la^{(i)}}
\qquad
\text{and}
\qquad
N = \bigoplus_{\ell \in \Lambda_2} \T_{b^{(\ell)}}\kbf_{\mu^{(\ell)}}
$$
where $|\la^{(i)}| \ne 0$ and $|\mu^{(\ell)}| \ne 0$ for all 
$i \in \Lambda_1$ and $\ell \in \Lambda_2$. If ${\rm Rank}^M = {\rm Rank}^N$, 
then the two multisets 
\begin{eqnarray}    \label{CorRk-inv-MN.01}
\Big \{\big (a^{(i)}, (\la^{(i)})_{i_1, \ldots, i_{d-1}} \big )
|(i_1, \ldots, i_{d-1}) \in \N^{d-1}, 
i \in \Lambda_1 \text{ and } (\la^{(i)})_{i_1, \ldots, i_{d-1}} > 0 \Big \}
\end{eqnarray}
and 
\begin{eqnarray}    \label{CorRk-inv-MN.02}
\Big  \{ \big (b^{(\ell)}, (\mu^{(\ell)})_{i_1, \ldots, i_{d-1}} \big )
|(i_1, \ldots, i_{d-1}) \in \N^{d-1}, \ell \in \Lambda_2 \text{ and } 
(\mu^{(\ell)})_{i_1, \ldots, i_{d-1}} > 0 \Big \}
\end{eqnarray}
are equal, and $\sum_{i \in \Lambda_1} |\la^{(i)}| 
= \sum_{\ell \in \Lambda_2} |\mu^{(\ell)}|$.
\end{corollary}
\begin{proof}
Follows immediately from Theorem~\ref{ThmRk-inv-MN}.
\end{proof}

\begin{remark} \label{remarkd=1RkInv}
Recall from Definition~\ref{d-part} that a $1$-dimensional partition is 
of the form $\la = (n)_O$ for some $n \in \N \sqcup \{+\infty\}$. 
Moreover, every $\N$-indexed persistence module admits a barcode. 
Therefore, when $d = 1$, Theorem~\ref{ThmRk-inv-MN} recovers the well-known result 
that the rank invariant and the barcode determine each other uniquely. 
Unfortunately, when $d > 1$, Example~\ref{example_MN} shows that 
the rank invariant of a decomposable $\N^d$-indexed persistence modules 
does not determine the barcode.
\end{remark}

\bigskip\noindent
{\bf Data Availability}

Data sharing is not applicable to this article as no datasets were generated or 
analyzed during the current study.

\bigskip\noindent
{\bf Conflict of Interest}

The authors declare that there is no conflict of interest.

\bigskip\noindent
{\bf Acknowledgment}

The authors would like to thank all the reviewers for reading 
the manuscript carefully and for providing insightful comments.

\end{document}